\documentclass[amstex,11pt,reqno]{amsart}

\usepackage[a4paper,margin=30mm]{geometry}
\usepackage[T1]{fontenc}
\usepackage[utf8]{inputenc}
\usepackage{lmodern}
\usepackage{amsmath,amssymb,amsthm,mathtools,mathrsfs,mathdots}
\usepackage{bm}
\usepackage{enumitem}
\usepackage{tikz}
\usetikzlibrary{calc,positioning,arrows.meta}
\usetikzlibrary{backgrounds, fit}
\usetikzlibrary{shapes.geometric}
\usepackage{graphicx}
\usepackage[colorlinks=true,linkcolor=blue,citecolor=blue,urlcolor=blue]{hyperref}
\usepackage[nameinlink,noabbrev]{cleveref}

\newtheorem{theorem}{Theorem}[section]
\newtheorem{proposition}[theorem]{Proposition}
\newtheorem{lemma}[theorem]{Lemma}
\newtheorem{corollary}[theorem]{Corollary}
\theoremstyle{definition}
\newtheorem{definition}[theorem]{Definition}
\newtheorem{remark}[theorem]{Remark}
\newtheorem{example}[theorem]{Example}

\crefname{theorem}{Theorem}{Theorems}
\Crefname{theorem}{Theorem}{Theorems}
\crefname{proposition}{Proposition}{Propositions}
\Crefname{proposition}{Proposition}{Propositions}
\crefname{lemma}{Lemma}{Lemmas}
\Crefname{lemma}{Lemma}{Lemmas}
\crefname{corollary}{Corollary}{Corollaries}
\Crefname{corollary}{Corollary}{Corollaries}
\crefname{definition}{Definition}{Definitions}
\Crefname{definition}{Definition}{Definitions}
\crefname{remark}{Remark}{Remarks}
\Crefname{remark}{Remark}{Remarks}

\newcommand{\Z}{\mathbb Z}
\newcommand{\cF}{\mathcal F}
\newcommand{\wt}{\operatorname{wt}}
\newcommand{\ev}{\operatorname{ev}_1}
\newcommand{\bx}{\bm{x}}
\newcommand{\by}{\bm{y}}
\newcommand{\ba}{\bm{a}}
\newcommand{\cR}{\mathscr R}
\newcommand{\cL}{\mathscr L}

\title[Conway--Coxeter theorem for decorated frieze]{A Conway--Coxeter theorem \\for decorated frieze patterns}
\author[T. Kuwana]{Takeru Kuwana}
\address{Department of Mathematics, School of Education, Waseda University, \endgraf 1-6-1 Nish-Waseda, Shinjuku, Tokyo 169-8050, Japan}
\email{kuwanatakeru0521@gmail.com}
\author[K. Matsuzaki]{Katsuhiko Matsuzaki}
\address{Department of Mathematics, School of Education, Waseda University, \endgraf 1-6-1 Nish-Waseda, Shinjuku, Tokyo 169-8050, Japan}
\email{matsuzak@waseda.jp}

\begin{document}

\begin{abstract}
Let $m\geq0$ and put $n=m+3$. We introduce a normalized positive Laurent class of decorated frieze patterns and prove a decorated analogue of the Conway--Coxeter classification theorem. Namely, such friezes of width $m$ are in canonical bijection with weighted triangulations of a convex $n$-gon, whose boundary edges are labeled by independent variables $y_1,\ldots,y_n$ and whose diagonals are labeled by $x_1,\ldots,x_m$. While a weighted Conway--Coxeter propagation algorithm constructs the frieze from a weighted triangulation, a main new ingredient is an explicit nonrecursive Laurent formula expressing each quiddity entry directly from the weighted triangles incident to the corresponding vertex. Conversely, specialization to $1$, Laurent positivity, and a decorated cutting-and-gluing procedure recover the weighted triangulation from the frieze. 
%We also prove that closed decorated friezes are automatically $n$-periodic and satisfy glide-reflection symmetry, without assuming positivity or periodicity.
\end{abstract}

\maketitle

\section{Introduction}\label{sec:intro}

Frieze patterns were introduced by Coxeter \cite{Coxeter1971} and studied in detail by Conway and Coxeter \cite{CC1,CC2}. A classical positive integral closed frieze is an array of positive integers bounded by rows of $1$'s and satisfying the unimodular rule
\[
 ad-bc=1
\]
in every elementary diamond. Two fundamental properties of the classical theory are often separated as follows. In the terminology used by Morier-Genoud \cite{Morier2Friezes}, if a closed frieze has width $m=n-3$, then
\begin{enumerate}[label=\textup{(CC\arabic*)},leftmargin=3.7em]
\item every row is $n$-periodic;
\item positive integral closed friezes are in one-to-one correspondence with triangulations of a convex $n$-gon, and the first nontrivial row records the numbers of incident triangles at the vertices.
\end{enumerate}
The second statement is the celebrated Conway--Coxeter classification theorem and will be the main point of comparison for the present paper. The first can also be proved directly from the second-order recurrence satisfied by the diagonals; see, for example, Henry \cite{Henry2013}. Classical closed friezes possess a further characteristic symmetry, namely glide-reflection symmetry. It can be read geometrically from the Broline--Crowe--Isaacs interpretation of frieze entries \cite{BCI}, but it is also intrinsic to the algebraic frieze relations. Here and throughout the paper, saying that a frieze is $n$-periodic means that $n$ is a period; it does not mean that $n$ is necessarily the least positive period.

The geometry of the classical correspondence was developed further by Broline, Crowe and Isaacs \cite{BCI}, who interpreted all frieze entries in terms of a triangulated polygon. The modern theory is closely related to cluster algebras, Laurent positivity, and perfect matchings. Propp \cite{Propp2020} described a perfect-matching model for type $A$ friezes and the associated Laurent polynomials. Frieze patterns with coefficients were studied systematically by Cuntz, Holm and J\o rgensen \cite{CHJ2020}; in that setting boundary entries vary and local determinant identities become Ptolemy-type relations with coefficients. Cuntz and Holm \cite{CH2021} later characterized numerical coefficient friezes arising from subpolygons of Conway--Coxeter friezes. The coefficient theory has also been extended to the noncommutative setting, including propagation, quiddity cycles, and reduction formulae \cite{CHJ2024}. Further geometric models include Farey triangulations and Farey complexes \cite{MOT2015,Short2025}.

Of particular relevance to the present paper is Nishiyama's treatment of weighted triangulations and decorated friezes \cite[Chapters~10--11]{Nishiyama2022}. His decorated version of the Conway--Coxeter algorithm starts with a triangulated polygon carrying weights on its boundary edges and diagonals and propagates these data through adjacent triangles. In this way the entries of the associated decorated frieze can be calculated successively. Thus there is already a natural construction in the direction
\[
 \text{weighted triangulation}
 \longrightarrow
 \text{decorated frieze}.
\]
One purpose of the present work is to complement this propagation procedure by a direct description of the first nontrivial (quiddity) row. We obtain an explicit formula for each decorated quiddity entry in terms of the weighted triangles incident to a single vertex. This formula is nonrecursive and may be viewed as the weighted counterpart of the classical rule saying that a quiddity entry is the number of triangles incident to the corresponding vertex.

We now describe our setting. Fix $m\geq0$ and set $n=m+3$. Let
\[
 \bx=(x_1,\ldots,x_m),\qquad
 \by=(y_1,\ldots,y_n)
\]
be algebraically independent variables, with the convention $y_{i+n}=y_i$. A closed decorated frieze has the same strip shape as a classical closed frieze, but the two nonzero boundary rows are decorated by the periodic sequence of $\by$-variables, and the local rule becomes
\begin{equation}\label{eq:intro-rule}
 W_{i,j}W_{i+1,j+1}-W_{i+1,j}W_{i,j+1}=y_i y_j.
\end{equation}
Periodicity is not included in the definition, but every interior entry is assumed to be nonzero.

Before turning to the classification theorem, we record an algebraic fact about closed decorated friezes. It is the analogue of (CC1).

\begin{theorem}[Decorated periodicity -- (DCC1)]\label{thm:intro-periodicity}
Every closed decorated frieze of width $m$ is $(m+3)$-periodic. More precisely, if $n=m+3$, then
\[
 W_{i+n,j+n}=W_{i,j}
\]
for all entries of the strip.

If the boundary decorations $y_1,\ldots,y_n$ are algebraically independent variables, then $n$ is the least positive period of the formal decorated frieze.
\end{theorem}

%This result requires neither positivity nor Laurent positivity. 
It requires no positivity
assumption; only the nonvanishing condition built into the definition of a closed decorated frieze is used.
It follows directly from the decorated second-order recurrence and the closing boundary conditions. We also prove algebraically the stronger glide-reflection relation
\[
 W_{i,j}=W_{j,i+n}.
\]
A determinant representation $W_{i,j}=\det(u_i,u_j)$ yields the full Ptolemy relation through the Pl\"ucker identity, and the glide reflection follows from the closing boundary rows. Thus the infinite strip is determined by the finite triangular fundamental set
$\{W_{i,j}:1\leq i<j\leq n\}$. These facts provide the algebraic background for the classification, but the main result of the paper is the decorated counterpart of (CC2).

For this purpose we impose positivity and a normalization. We consider decorated friezes whose entries belong to
\[
 \Z_{\geq0}\,
 [x_1^{\pm1},\ldots,x_m^{\pm1},
 y_1^{\pm1},\ldots,y_n^{\pm1}]
 \setminus\{0\}.
\]
The distinguished variables $x_1,\ldots,x_m$ are intended to represent the $m=n-3$ diagonals of a triangulation. Our normalization requires that the Laurent monomial interior entries are precisely these distinguished variables, with no unused $\bx$-variables.

Let $P_n$ be a convex $n$-gon with cyclically labeled vertices
$v_1,\ldots,v_n$. An $(\bx,\by)$-weighted triangulation means that the boundary edge $[v_i,v_{i+1}]$ has weight $y_i$, while the $m$ internal diagonals are labeled bijectively by $x_1,\ldots,x_m$. Our principal result is the following.

\begin{theorem}[Decorated Conway--Coxeter theorem -- (DCC2)]
\label{thm:intro-classification}
For every $m\geq0$ and $n=m+3$, there is a canonical bijection
\begin{equation*}
\begin{gathered}
 \{(\bx,\by)\text{-weighted triangulations $\Delta$ of }P_n\}\\
 \longleftrightarrow\\
 \{\text{normalized positive Laurent decorated friezes $\cF$
 of width }m\}.
\end{gathered}
\end{equation*}
Moreover, the quiddity entries are given directly by the explicit
geometric formula \eqref{eq:intro-quiddity}.
\end{theorem}

There are two directions in this correspondence, and it is useful to
distinguish them. Starting from a weighted triangulation, the decorated
Conway--Coxeter algorithm described by Nishiyama
\cite[Chapter~11]{Nishiyama2022} successively determines the entries of
the decorated frieze. Equivalently, the same entries can be obtained by
Ptolemy propagation or by the type $A$ perfect-matching formula. In
particular, the existence of a positive Laurent decorated frieze
associated with a weighted triangulation is already suggested by this
weighted propagation theory.

A new feature of the present paper in this direction is that the
quiddity need not be found by carrying out the propagation. For a
vertex $v$, let $E_v$ be the collection of all triangulation edges
incident to $v$, let $D_v\subset E_v$ be the internal diagonals
incident to $v$, and, for a triangle $T$ containing $v$, let
$E_v(T)$ denote the two sides of $T$ incident to $v$ and
$\operatorname{opp}_v(T)$ the side opposite $v$. We prove that the
quiddity entry associated with $v$ is
\begin{equation}\label{eq:intro-quiddity}
 Q_\Delta(v)=
 \frac{
 \displaystyle
 \sum_{\substack{T\in\Delta,\, v\in T}}
 \wt\bigl(\operatorname{opp}_v(T)\bigr)
 \prod_{e\in E_v\setminus E_v(T)}\wt(e)}
 {\displaystyle\prod_{d\in D_v}\wt(d)}.
\end{equation}
This is formula \eqref{eq:geometric-quiddity} in Section \ref{sec:geometry}. 
Its form is one of
the main ingredients of our classification theorem. 
%Each triangle
%incident to $v$ contributes exactly one term: the weight of the side
%opposite $v$ is multiplied by the weights of all other edges incident
%to $v$ that do not belong to that triangle, while the product of the
%incident diagonal weights forms the common denominator. Thus 
The
formula reads the quiddity directly from the weighted star of the
vertex, without computing any other frieze entries. When all weights
are specialized to $1$, every summand becomes $1$, and
\eqref{eq:intro-quiddity} reduces precisely to the classical
Conway--Coxeter rule counting the triangles incident to $v$.

The converse direction of (DCC2) is more subtle and constitutes the
main rigidity statement. Starting with an abstract normalized positive
Laurent decorated frieze, one must show that it actually comes from a
weighted triangulation. The local rule \eqref{eq:intro-rule} by itself
does not provide such a triangulation. Our argument first specializes
all variables $x_k$ and $y_i$ to $1$. The resulting array is an
ordinary positive integral Conway--Coxeter frieze and hence determines
a unique triangulation of the labeled polygon. Laurent positivity then
allows this triangulation to be recovered already before specialization:
the entries occupying its diagonals are exactly the monomial entries of
the decorated frieze. We call this collection the \emph{monomial
skeleton}. By normalization, its elements are precisely
$x_1,\ldots,x_m$.

This observation permits the classical cutting-and-gluing proof of
Conway and Coxeter to be lifted to the decorated setting. In the
classical theory an ear corresponds to a quiddity entry $1$, and
cutting off the ear replaces the quiddity sequence
\[
 (\ldots,a_{i-1},1,a_{i+1},\ldots)
 \quad\text{by}\quad
  (\ldots,a'_{i-1},a'_{i},\ldots)=(\ldots,a_{i-1}-1,a_{i+1}-1,\ldots).
\]
In our setting an ear is detected by a monomial quiddity entry
$a_i=x$, and the neighboring entries are replaced by
\begin{equation}\label{eq:intro-reduction}
 a'_{i-1}
 =\frac{a_{i-1}x-y_{i-1}y_{i+1}}{y_i},
 \qquad
 a'_i
 =\frac{x a_{i+1}-y_i y_{i+2}}{y_{i+1}}.
\end{equation}
A transfer-matrix identity shows that this reduction deletes exactly
one inverse $V$-shaped strip from the frieze and produces a decorated
frieze of one smaller width. On the geometric side, the same formulae
for the geometric quiddity sequence given by \eqref{eq:intro-quiddity}
are obtained by removing the ear whose opposite diagonal has weight
$x$. Conversely, attaching the ear gives the inverse transformation.
The algebraic and geometric inductions therefore agree step by step.

The explicit quiddity formula plays the
important role in this comparison. Nishiyama's propagation algorithm
shows how a weighted triangulation generates the complete decorated
frieze, whereas \eqref{eq:intro-quiddity} identifies, in closed form,
the quiddity data that drive the second-order recurrence. In this
sense, the formula supplies a direct weighted analogue of the most
geometric part of the classical Conway--Coxeter theorem: the first
nontrivial row can be read immediately from the triangulation itself.
Together with the converse reconstruction from the monomial skeleton,
this yields the bijection in (DCC2).

It is useful to distinguish this classification statement from the
broader theory of friezes with coefficients. The local decorated
unimodular relation, Ptolemy propagation, and related reduction
formulae have analogues in the coefficient theory
\cite{CHJ2020,CHJ2024}. What is specific here is the rigidity of the
normalized Laurent-positive symbolic class: the abstract frieze itself
determines a unique weighted triangulation, while the weighted
triangulation determines not only the full frieze by propagation but
also its quiddity directly through
\eqref{eq:intro-quiddity}. Under the specialization
$x_1=\cdots=x_m=y_1=\cdots=y_n=1$, this correspondence becomes exactly
the classical Conway--Coxeter theorem.

The paper is organized as follows. In Section \ref{sec:periodicity} we
define closed decorated friezes without assuming periodicity, derive
the decorated difference equation, and establish periodicity,
the determinant representation, the full Ptolemy relation, and
glide-reflection symmetry. In Section \ref{sec:positive} we introduce
positive Laurent decorated friezes and the normalization condition,
and identify their monomial skeleton by specialization to the
classical theory. In Section \ref{sec:geometry} we study weighted
triangulations and prove the explicit geometric quiddity formula
\eqref{eq:intro-quiddity}. In Section \ref{sec:reduction} we
develop the algebraic reduction at a monomial quiddity entry and show
that it agrees exactly with weighted ear removal. Finally, in Section
\ref{sec:main} we combine these ingredients to prove the decorated
Conway--Coxeter classification theorem (DCC2).

\section{Closed decorated friezes and algebraic periodicity}\label{sec:periodicity}

\subsection{Algebraic periodicity}
For $n=m+3$, $m \geq 0$, put
\[
 \cR_m=\Z\,[x_1^{\pm1},\ldots,x_m^{\pm1},y_1^{\pm1},\ldots,y_n^{\pm1}],
\]
and extend the boundary variables by the convention
\[
 y_{i+n}=y_i\qquad (i\in\Z).
\]
The periodic extension of the \emph{boundary decoration} is part of the data; no periodicity of the interior entries is assumed.
Here, the interior entries refer to those in the stripe of width $m$ between the $2$-nd and $(n-2)$-th rows in 
Figure \ref{fig:frieze}.

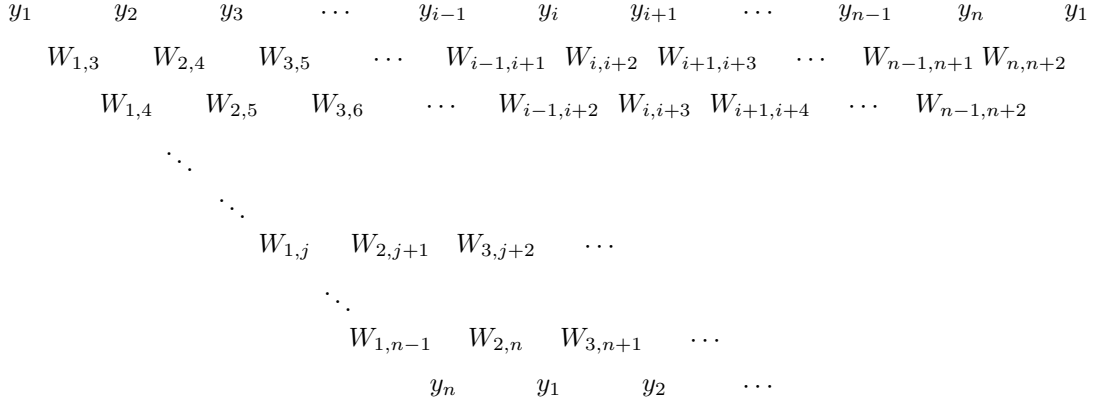
\begin{figure}[h]
\centering
\resizebox{0.97\textwidth}{!}{%
\begin{tikzpicture}[x=1.4cm,y=0.62cm,font=\small]
\def\rows{
  {$y_1$, $y_2$, $y_3$, $\cdots$, $y_{i-1}$, $y_i$, $y_{i+1}$, $\cdots$, $y_{n-1}$, $y_n$, $y_1$},
  {$W_{1,3}$, $W_{2,4}$, $W_{3,5}$, $\cdots$, $W_{i-1,i+1}$, $W_{i,i+2}$, $W_{i+1,i+3}$, $\cdots$, $W_{n-1,n+1}$, $W_{n,n+2}$},
  {$W_{1,4}$, $W_{2,5}$, $W_{3,6}$, $\cdots$, $W_{i-1,i+2}$, $W_{i,i+3}$, $W_{i+1,i+4}$,$\cdots$, $W_{n-1,n+2}$}, 
  {$\ddots$, $\ $, $\ $},
  {$\ddots$, },
  {$W_{1,j}$, $W_{2,j+1}$, $W_{3,j+2}$, $\cdots$},
  {$\ddots$},
  {$W_{1,n-1}$, $W_{2,n}$, $W_{3,n+1}$, $\cdots$},
  {$y_n$, $y_1$, $y_2$, $\cdots$},
}
\foreach \row [count=\r from 0] in \rows {
  \foreach \val [count=\c from 0] in \row {
    \node at ({\c+0.5*\r},{-\r}) {\val};
  }
}
\end{tikzpicture}}
\caption{The arrangement of a closed decorated frieze of width $m$.}
\label{fig:frieze}
\end{figure}

\begin{definition}\label{def:closed-decorated}
A \emph{closed decorated frieze of width $m$} is a frieze-shaped strip $\cF=(W_{i,j})$ with entries in $\cR_m$, 
indexed by $i,j\in\Z$ with $0 \leq j-i \leq n$,
and satisfying
\begin{enumerate}[label=\textup{(D\arabic*)},leftmargin=2.6em]
\item the zero boundary rows (the $0$-th and $n$-th rows not appearing in Figure \ref{fig:frieze}) are
\[
 W_{i,i}=0,\qquad W_{i,i+n}=0;
\]
\item the two nonzero boundary rows (the $1$-st and $(n-1)$-th rows in Figure \ref{fig:frieze}) are
\[
 W_{i,i+1}=y_i,\qquad W_{i,i+n-1}=y_{i-1};
\]
\item every interior entry is nonzero;
\item every elementary diamond in the strip satisfies
\begin{equation}\label{eq:diamond}
 W_{i,j}W_{i+1,j+1}-W_{i+1,j}W_{i,j+1}
 =y_i y_j
 =W_{i,i+1}W_{j,j+1}.
\end{equation}
\end{enumerate}
No horizontal periodicity is included in the definition.
\end{definition}

\begin{remark}
The nonvanishing assumption in {\rm (D3)} is convenient rather than
essential. It is said that $\mathcal F$ is \emph{tame} if
$$
\det
\begin{pmatrix}
W_{i,j}&W_{i+1,j}&W_{i+2,j}\\
W_{i,j+1}&W_{i+1,j+1}&W_{i+2,j+1}\\
W_{i,j+2}&W_{i+1,j+2}&W_{i+2,j+2}
\end{pmatrix}=0.
$$
Since $\mathcal R_m$ is an integral domain, {\rm (D3)} implies
tameness by the standard argument; compare
\cite[Proposition~2.6]{CHJ2020}.  Then, the Ptolemy propagation relations
of \cite[Theorem~3.3]{CHJ2020} apply. For the positive Laurent friezes considered from Section~3 onward,
{\rm (D3)} is automatic.
\end{remark}

The \emph{quiddity row} is the bi-infinite sequence
\[
 a_i=W_{i,i+2}\qquad(i\in\Z).
\]
At this stage it is not assumed to be periodic.

For each $s\in\Z$, consider the diagonal starting from the boundary entry $y_s$. Set
\begin{equation*}\label{eq:diag-def}
 f^{(s)}_{-1}=0,
 \qquad
 f^{(s)}_r=W_{s,s+r+1}\quad(0\leq r\leq n-1).
\end{equation*}
We call this the $s$-th diogonal. Thus
\[
 f^{(s)}_0=y_s,
 \qquad
 f^{(s)}_1=a_s,
 \qquad
 f^{(s)}_{n-2}=y_{s-1},
 \qquad
 f^{(s)}_{n-1}=0.
\]

\begin{proposition}[Decorated difference equation]\label[proposition]{prop:recurrence}
For every $s\in\Z$, the diagonal sequence satisfies
\begin{equation}\label{eq:general-rec}
 f^{(s)}_{r+1}
 =\frac{a_{s+r}}{y_{s+r}}f^{(s)}_r
 -\frac{y_{s+r+1}}{y_{s+r}}f^{(s)}_{r-1}
 \qquad(0\leq r\leq n-2).
\end{equation}
The same recurrence continues algebraically one step beyond the lower zero boundary and gives
\begin{equation*}\label{eq:anti-end}
 f^{(s)}_n=-y_s.
\end{equation*}
\end{proposition}

\begin{proof}
The local coefficient rule implies the Ptolemy propagation relation; see \cite[Theorem~3.3]{CHJ2020}. Applied to the four relevant boundary vertices, it gives
\[
 a_{s+r}f^{(s)}_r
 =y_{s+r}f^{(s)}_{r+1}+y_{s+r+1}f^{(s)}_{r-1},
\]
which is equivalent to \eqref{eq:general-rec}. This relation can also be obtained directly by the standard determinant subtraction used in the classical proof of the frieze recurrence.

To make this determinant subtraction explicit, let us compare the diagonal
$f^{(s)}$ with the adjacent diagonal $f^{(s+1)}$.  For $r\geq1$, the
diamond relation gives
\[
 f^{(s)}_r f^{(s+1)}_r
 -
 f^{(s)}_{r+1} f^{(s+1)}_{r-1}
 =
 y_s y_{s+r+1},
\]
and, after shifting the index by one,
\[
 f^{(s)}_{r-1} f^{(s+1)}_{r-1}
 -
 f^{(s)}_r f^{(s+1)}_{r-2}
 =
 y_s y_{s+r}.
\]
Equivalently,
\[
 \frac{1}{y_{s+r+1}}
 \begin{vmatrix}
 f^{(s)}_r & f^{(s+1)}_{r-1}\\
 f^{(s)}_{r+1} & f^{(s+1)}_r
 \end{vmatrix}
 =
 y_s
=
 \frac{1}{y_{s+r}}
 \begin{vmatrix}
 f^{(s)}_{r-1} & f^{(s+1)}_{r-2}\\
 f^{(s)}_r & f^{(s+1)}_{r-1}
 \end{vmatrix}.
\]
Subtracting these two identities yields
\begin{equation}\label{eq:det-subtraction}
 \begin{vmatrix}
 f^{(s)}_r & f^{(s+1)}_{r-1}\\[2mm]
 \displaystyle
 \frac{f^{(s)}_{r+1}}{y_{s+r+1}}
 +\frac{f^{(s)}_{r-1}}{y_{s+r}}
 &
 \displaystyle
 \frac{f^{(s+1)}_r}{y_{s+r+1}}
 +\frac{f^{(s+1)}_{r-2}}{y_{s+r}}
 \end{vmatrix}
 =0.
\end{equation}

This calculation can also be organized as an induction on $r$.  The case
$r=0$ of \eqref{eq:general-rec} is immediate from
$f^{(s)}_{-1}=0$, $f^{(s)}_0=y_s$, and $f^{(s)}_1=a_s$.
Assume that the recurrence has already been established at level $r-1$
for every starting index.  Applying it to the adjacent diagonal
$f^{(s+1)}$ gives
%\[
% f^{(s+1)}_r
% =
% \frac{a_{s+r}}{y_{s+r}}f^{(s+1)}_{r-1}
% -
% \frac{y_{s+r+1}}{y_{s+r}}f^{(s+1)}_{r-2},
%\]
%or, equivalently,
\begin{equation}\label{eq:adjacent-recurrence-rearranged}
 \frac{f^{(s+1)}_r}{y_{s+r+1}}
 +
 \frac{f^{(s+1)}_{r-2}}{y_{s+r}}
 =
 \frac{a_{s+r}}{y_{s+r}y_{s+r+1}}
 f^{(s+1)}_{r-1}.
\end{equation}
Substituting \eqref{eq:adjacent-recurrence-rearranged} into
\eqref{eq:det-subtraction} and expanding the determinant gives
\[
 f^{(s+1)}_{r-1}
 \left(
 \frac{a_{s+r}}{y_{s+r}y_{s+r+1}}f^{(s)}_r
 -
 \frac{f^{(s)}_{r+1}}{y_{s+r+1}}
 -
 \frac{f^{(s)}_{r-1}}{y_{s+r}}
 \right)
 =0.
\]
Since the factor
$f^{(s+1)}_{r-1}$ is nonzero by {\rm (D3)}, 
\[
 \frac{f^{(s)}_{r+1}}{y_{s+r+1}}
 =
 \frac{a_{s+r}}{y_{s+r}y_{s+r+1}}f^{(s)}_r
 -
 \frac{f^{(s)}_{r-1}}{y_{s+r}}.
\]
Multiplying by $y_{s+r+1}$ gives exactly \eqref{eq:general-rec}.
%\[
% f^{(s)}_{r+1}
% =
% \frac{a_{s+r}}{y_{s+r}}f^{(s)}_r
% -
% \frac{y_{s+r+1}}{y_{s+r}}f^{(s)}_{r-1}.
%\]
%Thus the usual determinant subtraction explains directly how the
%second-order decorated recurrence propagates from one diagonal to the
%next.  In the general algebraic setting of the present proposition,
%the Ptolemy argument above establishes the same identity without any
%nonvanishing assumption on the interior entries.

For $r=n-1$, use $f^{(s)}_{n-1}=0$, $f^{(s)}_{n-2}=y_{s-1}=y_{s+n-1}$, and $y_{s+n}=y_s$. The same recurrence gives
\[
 f^{(s)}_n
 =-\frac{y_{s+n}}{y_{s+n-1}}f^{(s)}_{n-2}
 =-y_s.
\]
This completes the proof.
\end{proof}

For fixed boundary data and quiddity, the recurrence determines every diagonal and hence the whole frieze.

\begin{corollary}\label[corollary]{cor:unique}
For fixed boundary data $\by$, there is at most one closed decorated frieze with a prescribed quiddity sequence $\ba$.
\end{corollary}

\begin{proof}
This follows immediately by iterating the recurrence \eqref{eq:general-rec} from the upper boundary row.
\end{proof}

Define the transfer matrices
\begin{equation*}\label{eq:A}
 A_i=
 \begin{pmatrix}
 a_i/y_i&-y_{i+1}/y_i\\
 1&0
 \end{pmatrix}.
\end{equation*}
Then \eqref{eq:general-rec} is equivalent to
\begin{equation*}\label{eq:Arec}
 \binom{f^{(s)}_{r+1}}{f^{(s)}_r}
 =A_{s+r}
 \binom{f^{(s)}_r}{f^{(s)}_{r-1}}.
\end{equation*}
Notice that
\begin{equation}\label{eq:detA}
 \det A_i=\frac{y_{i+1}}{y_i}.
\end{equation}

\begin{theorem}[Algebraic periodicity]\label{thm:periodicity}
Let $\cF$ be a closed decorated frieze of width $m$, and put $n=m+3$. Then
$a_{i+n}=a_i$ $(i\in\Z)$,
and consequently
$W_{i+n,j+n}=W_{i,j}$
throughout the strip. Thus $n$ is a period of every closed decorated frieze of width $m$.
\end{theorem}

\begin{proof}
For each $s\in\Z$, set
\begin{equation*}\label{eq:Ps}
 P_s=A_{s+n-1}A_{s+n-2}\cdots A_s.
\end{equation*}
By \cref{prop:recurrence},
\[
 P_s\binom{y_s}{0}=\binom{-y_s}{0}.
\]
Hence the first column of $P_s$ is $(-1,0)^t$. Moreover, by \eqref{eq:detA} and the $n$-periodicity of the boundary decorations,
\[
 \det P_s
 =\prod_{r=0}^{n-1}\frac{y_{s+r+1}}{y_{s+r}}
 =1.
\]
Therefore
\begin{equation}\label{eq:triangular-P}
 P_s=
 \begin{pmatrix}-1&c_s\\0&-1\end{pmatrix}
\end{equation}
for some $c_s\in\cR_m$.

The products for two consecutive starting points are related by
\begin{equation}\label{eq:conjugation}
 P_{s+1}=A_{s+n}P_sA_s^{-1}.
\end{equation}
Since $y_{s+n}=y_s$ and $y_{s+n+1}=y_{s+1}$, a direct multiplication shows that the lower-left entry of the right-hand side of \eqref{eq:conjugation} is
$-c_s y_s/y_{s+1}$.
But $P_{s+1}$ also has the triangular form \eqref{eq:triangular-P}, so this lower-left entry must vanish. Since the $y_i$ are units, $c_s=0$. Thus
$P_s=-I$ for every $s\in\Z$.
Using this and \eqref{eq:conjugation} again, we obtain
\[
 -I=P_{s+1}=A_{s+n}(-I)A_s^{-1},
\]
so
$A_{s+n}=A_s$.
The boundary part of these matrices is already $n$-periodic; comparing the upper-left entries gives
$a_{s+n}=a_s$.

Finally, each diagonal is uniquely generated from its boundary initial data by \eqref{eq:general-rec}. Since both the boundary decorations and the quiddity row are $n$-periodic, shifting all indices by $n$ gives the same recurrence with the same initial data. Hence $W_{i+n,j+n}=W_{i,j}$ for every entry.
\end{proof}

\begin{corollary}
%[Least period in the formal setting]
\label{cor:least-period}
For the formal decorated friezes considered here, with algebraically independent and pairwise distinct boundary variables $y_1,\ldots,y_n$, the least positive horizontal period is exactly $n$.
\end{corollary}

\begin{proof}
If $p>0$ were a horizontal period, then the upper boundary row would satisfy $y_{i+p}=y_i$ for every $i$. Since the $y_1,\ldots,y_n$ are distinct independent variables and are extended with period $n$, this is possible only when $n$ divides $p$. By \cref{thm:periodicity}, $n$ itself is a period.
\end{proof}

\begin{remark}\label{rem:period-specialization}
The least-period assertion is specific to the formal labeled setting. After specializing some boundary variables, a smaller period may occur. This is already familiar in the classical theory: a closed frieze associated with an $n$-gon always has $n$ as a period, but its least period can be smaller when the corresponding triangulation has rotational symmetry. Thus the invariant statement parallel to classical (CC1) is the first assertion of \cref{thm:periodicity}.
\end{remark}

\subsection{Determinant representation and glide-reflection symmetry}\label{subsec:glide}
The preceding proof deliberately establishes periodicity first. We now show that the same closed decorated frieze has a stronger symmetry. The argument is independent of positivity and of the polygonal classification.

Choose two vectors $u_0,u_1\in\cR_m^2$ such that
\[
 \det(u_0,u_1)=y_0,
\]
where $y_0=y_n$, and extend them in both directions by the vector recurrence
\begin{equation}\label{eq:vector-recurrence}
 u_{i+2}=\frac{a_i}{y_i}u_{i+1}-\frac{y_{i+1}}{y_i}u_i
 \qquad(i\in\Z).
\end{equation}
Since every $y_i$ is a unit in $\cR_m$, this determines a bi-infinite sequence of vectors.

\begin{proposition}[Determinant representation]\label[proposition]{prop:det-representation}
For all $i,j\in\Z$ with $0\le j-i\le n$,
\begin{equation*}\label{eq:det-representation}
 W_{i,j}=\det(u_i,u_j).
\end{equation*}
In particular,
\[
 \det(u_i,u_{i+1})=y_i,
 \qquad
 \det(u_i,u_{i+2})=a_i.
\]
\end{proposition}

\begin{proof}
Starting from $\det(u_0,u_1)=y_0$, recurrence \eqref{eq:vector-recurrence} gives
\[
 \det(u_{i+1},u_{i+2})
 =-\frac{y_{i+1}}{y_i}\det(u_{i+1},u_i)
 =y_{i+1}.
\]
Thus $\det(u_i,u_{i+1})=y_i$ for every $i$, and then
\[
 \det(u_i,u_{i+2})
 =\frac{a_i}{y_i}\det(u_i,u_{i+1})
 =a_i.
\]
Fix $i$ and put $V_{i,j}=\det(u_i,u_j)$. As a function of $j$, the sequence $V_{i,j}$ satisfies exactly the decorated difference equation of \cref{prop:recurrence}, with initial values
\[
 V_{i,i}=0,
 \qquad
 V_{i,i+1}=y_i.
\]
The same is true for $W_{i,j}$. Uniqueness for the second-order recurrence therefore gives $V_{i,j}=W_{i,j}$ throughout the closed strip.
\end{proof}

The determinant representation immediately upgrades the local diamond relation to the full Ptolemy relation.

\begin{corollary}[Ptolemy relation]\label[corollary]{cor:full-ptolemy}
For any $i<j<k<\ell$ for which all entries involved lie in the closed strip,
\begin{equation}\label{eq:full-ptolemy}
 W_{i,k}W_{j,\ell}
 =W_{i,j}W_{k,\ell}+W_{i,\ell}W_{j,k}.
\end{equation}
\end{corollary}

\begin{proof}
Apply the Pl\"ucker identity in dimension two,
\[
 \det(u_i,u_k)\det(u_j,u_\ell)
 =\det(u_i,u_j)\det(u_k,u_\ell)
  +\det(u_i,u_\ell)\det(u_j,u_k),
\]
and use \cref{prop:det-representation}.
\end{proof}

\begin{theorem}[Glide-reflection symmetry]\label{thm:glide}
Let $\cF$ be a closed decorated frieze of width $m$ and put $n=m+3$. Then
\begin{equation}\label{eq:glide}
 W_{i,j}=W_{j,i+n}
\end{equation}
for every $i,j\in\Z$ with $0\le j-i\le n$.
Equivalently, the index transformation
\[
 G(i,j)=(j,i+n)
\]
leaves the decorated frieze invariant. In the usual planar display of a frieze, $G$ is a glide reflection.
\end{theorem}

\begin{proof}
The cases $j=i$, $j=i+n-1$, and $j=i+n$ follow directly from the boundary conditions. Assume
$i<j<i+n-1$.
Apply \cref{cor:full-ptolemy} to the four indices
\[
 i<j<i+n-1<i+n.
\]
We obtain
\[
 W_{i,i+n-1}W_{j,i+n}
 =W_{i,j}W_{i+n-1,i+n}
  +W_{i,i+n}W_{j,i+n-1}.
\]
By closedness,
\[
 W_{i,i+n}=0,
 \qquad
 W_{i,i+n-1}=y_{i-1},
 \qquad
 W_{i+n-1,i+n}=y_{i+n-1}=y_{i-1}.
\]
Hence
\[
 y_{i-1}W_{j,i+n}=y_{i-1}W_{i,j}.
\]
Since $y_{i-1}$ is a unit in $\cR_m$, cancellation gives \eqref{eq:glide}.
\end{proof}

\begin{remark}
%[Anti-periodic vector lift]
\label{rem:vector-antiperiodic}
In the determinant model, the transfer-matrix identity $P_s=-I$ 
%from \eqref{eq:minusI} 
is equivalent to
$u_{i+n}=-u_i$ $(i\in\Z)$.
Thus the glide relation also has the compact determinant interpretation
\[
 W_{j,i+n}=\det(u_j,u_{i+n})
 =-\det(u_j,u_i)
 =\det(u_i,u_j)=W_{i,j}.
\]
This makes explicit the relation between the anti-periodicity of the underlying second-order system and the glide-reflection symmetry of the frieze.
\end{remark}

%\begin{figure}[h]
%\centering
%\resizebox{0.97\textwidth}{!}{%
%\begin{tikzpicture}[x=1.4cm,y=0.62cm,font=\small]
%\def\rows{
%  {$y_1$, $y_2$, $y_3$, $\cdots$, $\cdots$, $\cdots$, $y_{n-1}$, $y_n$}, 
%  {$W_{1,3}$, $W_{2,4}$, $W_{3,5}$, $\cdots$, $\cdots$, $W_{n-2,n}$, $W_{n-1,1}$, $W_{n,2}$},
%  {$W_{1,4}$, $W_{2,5}$, $W_{3,6}$, $\cdots$, $W_{n-3,n}$, $W_{n-2,1}$, $W_{n-1,2}$, $W_{n,3}$}, 
%  {$\ddots$, $\ $, $\ $, $\iddots$, $\iddots$, $\ $, $\ $, $\ddots$},
%  {$\ddots$, $\ $, $\iddots$, $\iddots$, $\ $, $\ $, $\ $, $\ddots$},
%  {$W_{1,n-1}$, $W_{2,n}$, $W_{3,1}$, $\cdots$, $\cdots$, $\cdots$, $W_{n-1,n-3}$, $W_{n,n-2}$},
%  {$y_n$, $y_1$, $y_2$, $\cdots$, $\cdots$, $\cdots$, $y_{n-2}$, $y_{n-1}$},
%}
%\foreach \row [count=\r from 0] in \rows {
%  \foreach \val [count=\c from 0] in \row {
%    \node at ({\c+0.5*\r},{-\r}) {\val};
%  }
%}
%\end{tikzpicture}}
%\caption{The glide reflection and the fundamental domain.}
%\label{fig:frieze}
%\end{figure}

\begin{figure}[h]
\centering
\resizebox{0.97\textwidth}{!}{%
\begin{tikzpicture}[x=1.4cm,y=0.62cm,font=\small]
\def\rows{
  {$y_1$, $y_2$, $y_3$, $\cdots$, $\cdots$, $\cdots$, $y_{n-1}$, $y_n$}, 
  {$W_{1,3}$, $W_{2,4}$, $W_{3,5}$, $\cdots$, $\cdots$, $W_{n-2,n}$, $W_{n-1,1}$, $W_{n,2}$},
  {$W_{1,4}$, $W_{2,5}$, $W_{3,6}$, $\cdots$, $W_{n-3,n}$, $W_{n-2,1}$, $W_{n-1,2}$, $W_{n,3}$}, 
  {$\ddots$, $\ $, $\ $, $\iddots$, $\iddots$, $\ $, $\ $, $\ddots$},
  {$W_{1,n-2}$, $\ $, $W_{3,n}$, $W_{4,1}$, $\ $, $\ $, $\ $, $W_{n,n-3}$},
  {$W_{1,n-1}$, $W_{2,n}$, $W_{3,1}$, $\cdots$, $\cdots$, $\cdots$, $W_{n-1,n-3}$, $W_{n,n-2}$},
  {$y_n$, $y_1$, $y_2$, $\cdots$, $\cdots$, $\cdots$, $y_{n-2}$, $y_{n-1}$},
}
\foreach \row [count=\r from 0] in \rows {
  \foreach \val [count=\c from 0] in \row {
    % 各ノードに名前に 「n-行番号-列番号」 という名前を自動付与します
    \node (n-\r-\c) at ({\c+0.5*\r},{-\r}) {\val};
  }
}

% --- 背景に三角形を配置 ---
\begin{scope}[on background layer]
    \filldraw[
        fill=yellow!30,       
        draw=orange,          
        thick,                
        rounded corners=6pt   % 大きくした分、角の丸みも少し増やすと綺麗です
    ] 
    % 各頂点を外側に 3mm〜5mm ほどずらします
    ([shift={(-4mm, 1mm)}]n-0-0.north west) -- 
    ([shift={(5mm, 1mm)}]n-0-6.north east) -- 
    ([shift={(0mm, -3mm)}]n-6-0.south) -- cycle;
%% --- 背景に三角形を配置 ---
%\begin{scope}[on background layer]
%    \filldraw[
%        fill=yellow!30,       % 内部の塗りつぶし色（薄い黄色）
%        draw=orange,          % 囲み線の色
%        thick,                % 線の太さ
%        rounded corners=4pt   % 角を少し丸くして綺麗にする
%    ] 
%    % (n-0-0) が上段1番目の y_1
%    % (n-0-6) が上段7番目の y_{n-1}
%    % (n-6-0) が下段1番目の y_n (最下段は \r=6)
%    (n-0-0.north west) -- (n-0-6.north east) -- (n-6-0.south) -- cycle;
\end{scope}

\end{tikzpicture}}
\caption{The glide-reflection symmetry and the fundamental domain.}
\label{fig:symmetry}
\end{figure}
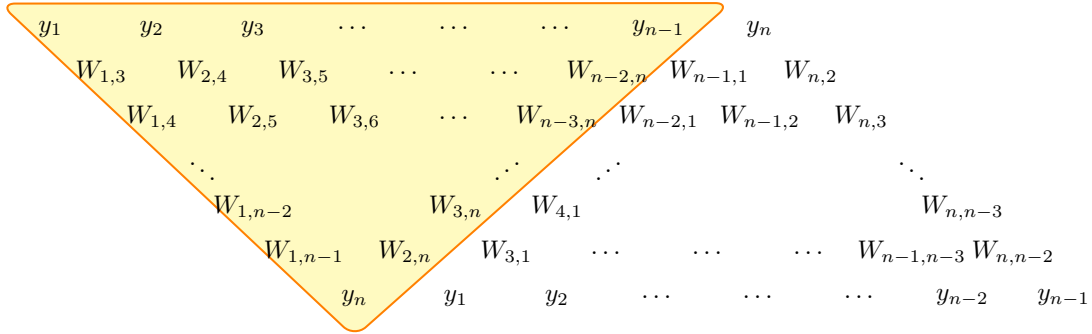

\begin{corollary}[Triangular fundamental domain]\label[corollary]{cor:fundamental-domain}
Every closed decorated frieze of width $m$ is completely determined by the finite triangular collection
\begin{equation*}\label{eq:fundamental-domain}
 \mathscr D_n=\{W_{i,j}:1\le i<j\le n\}.
\end{equation*}
Indeed, every nonzero entry of the infinite strip is carried to a unique entry of $\mathscr D_n$ by iterating the glide reflection $G$ (equivalently, by using $n$-periodicity together with \eqref{eq:glide}).
\end{corollary}

\begin{proof}
Let $W_{i,j}$ be nonzero, so $1\le j-i\le n-1$. By \cref{thm:periodicity}, translate both indices by a multiple of $n$ so that $1\le i\le n$. If then $j\le n$, the entry belongs to $\mathscr D_n$. If $j>n$, the glide relation gives
\[
 W_{i,j}=W_{j,i+n}=W_{j-n,i},
\]
where the last equality uses $n$-periodicity. Since $1\le j-n<i\le n$, this lies in $\mathscr D_n$. The reverse reconstruction is immediate from periodicity and the glide relation. Uniqueness of the representative follows because, modulo the translation $G^2(i,j)=(i+n,j+n)$, a $G$-orbit is determined by the unordered pair of residue classes $\{i,j\}$ modulo $n$, and $\mathscr D_n$ contains exactly one ordered representative $1\le i<j\le n$ of each such pair.
\end{proof}

\begin{remark}
%[Glide symmetry also implies periodicity]
\label{rem:glide-implies-period}
We have intentionally proved periodicity first, by the transfer-matrix argument of \cref{thm:periodicity}, and only afterwards established the stronger glide-reflection symmetry. Nevertheless, the logical implication can also be reversed once \eqref{eq:glide} is known: since
\[
 G^2(i,j)=(i+n,j+n),
\]
applying the glide reflection twice gives
\[
 W_{i+n,j+n}=W_{i,j}.
\]
Thus $n$-periodicity is also an immediate consequence of glide-reflection symmetry.
\end{remark}

\begin{remark}
%[Comparison with the classical theory]
\label{rem:classical-glide}
For classical positive integral friezes, glide-reflection symmetry can be obtained from the Broline--Crowe--Isaacs geometric interpretation of all frieze entries \cite{BCI}; it is also one of the standard algebraic symmetries in Coxeter's theory. The proof above shows that the same symmetry persists for closed decorated friezes before any positivity, normalization, or triangulation theorem is imposed. Thus the decorated theory retains not only the classical periodicity (CC1), but also the associated glide-reflection structure and its finite fundamental domain.
\end{remark}

\section{Positive Laurent decorated friezes and the monomial skeleton}\label{sec:positive}
From now on, let the set of positive Laurent polynomials be
\[
 \cL_m=
 \Z_{\geq0}[x_1^{\pm1},\ldots,x_m^{\pm1},y_1^{\pm1},\ldots,y_n^{\pm1}]\setminus\{0\}.
\]

\begin{definition}\label{def:positive-normalized}
A closed decorated frieze of width $m$ is called a \emph{positive Laurent decorated frieze} 
if every nonzero entry belongs to $\cL_m$. By \cref{thm:periodicity}, it is automatically $n$-periodic.

Such a frieze is \emph{normalized with respect to $\bx$} if
%, in the triangular fundamental domain $\mathscr D_n$ of \cref{cor:fundamental-domain},
\begin{enumerate}[label=\textup{(N\arabic*)},leftmargin=2.6em]
\item every interior Laurent monomial entry is one of $x_1,\ldots,x_m$;
\item every $x_k$ occurs as an interior entry.
\end{enumerate}
\end{definition}

\begin{remark}\label{rem:dummy}
Condition (N1) is the essential normalization. Condition (N2) merely excludes unused dummy variables and makes the bijection with triangulations whose $m$ diagonals are distinctly labeled by $x_1,\ldots,x_m$ literal. Once the classical skeleton is identified below, there are exactly $m$ diagonal positions, so (N1)--(N2) imply that each $x_k$ occurs exactly once in a fundamental region.
\end{remark}

Because periodicity has now been proved rather than assumed, the quiddity sequence may be written unambiguously as
\[
 \ba=(a_1,\ldots,a_n),\qquad a_i=W_{i,i+2}.
\]
We use the polygon convention that $W_{i,j}$ corresponds to the chord $[v_i,v_j]$. Consequently, $a_i$ is geometrically attached to the intermediate vertex $v_{i+1}$.

%For later use, write the first two diagonals in one period as
%\[
% f_{-1}=0,\qquad f_0=y_1,\qquad f_r=W_{1,r+2}\ (1\le r\le m),
% \qquad f_{m+1}=y_n,
%\]
%and
%\[
% g_{-1}=0,\qquad g_0=y_2,\qquad g_r=W_{2,r+3}\ (1\le r\le m),
% \qquad g_{m+1}=y_1.
%\]
%They satisfy the shifted forms of \eqref{eq:general-rec}.

%\begin{figure}[t]
%\centering
%\resizebox{0.97\textwidth}{!}{%
%\begin{tikzpicture}[x=1.15cm,y=0.62cm,font=\small]
%\def\rows{
%  {$y_1$, $y_2$, $y_3$, $\cdots$, $y_{i-1}$, $y_i$, $y_{i+1}$, $y_{i+2}$, $\cdots$, $y_{n-1}$, $y_n$, $y_1$},
%  {$a_1$, $a_2$, $\cdots$, $a_{i-2}$, $a_{i-1}$, $a_i$, $a_{i+1}$, $\cdots$, $a_{n-2}$, $a_{n-1}$, $a_n$},
%  {$f_2$, $g_2$, $\cdots$, $\iddots$, $\cdots$, $\ddots$, $\cdots$},
%  {$\ddots$, $\ $, $\ $},
%  {$f_{i-2}$, $\iddots$},
%  {$f_{i-1}$, $\iddots$},
%  {$f_i$},
%  {$\ddots$},
%  {$\ $},
%  {$f_m$, $g_m$, $h_m$, $\cdots$},
%  {$y_n$, $y_1$, $y_2$, $\cdots$},
%}
%\foreach \row [count=\r from 0] in \rows {
%  \foreach \val [count=\c from 0] in \row {
%    \node at ({\c+0.5*\r},{-\r}) {\val};
%  }
%}
%\end{tikzpicture}}
%\caption{The arrangement of a positive Laurent decorated frieze after periodicity has been established. The second displayed row is the quiddity sequence. The sequences $f,g,h,\ldots$ run along the diagonals.}
%\label{fig:frieze}
%\end{figure}

For a positive Laurent decorated frieze $\cF(\bx,\by)=(W_{i,j})$ with independent variables $\bx$ and $\by$,
we define the specialization map
\[
 \ev(\cF(\bx,\by))=\cF({\bf 1},{\bf 1}),\qquad {\bf 1}=(1\ldots,1).
\]
%The coefficients of $\ev(\cF(\by))$ is also denoted by $\ev(W_{i,j})$.
Applying $\ev$ entrywise to a positive Laurent decorated frieze, which is also denoted by $\ev(W_{i,j})$,
 turns \eqref{eq:diamond} into the ordinary unimodular rule. By \cref{thm:periodicity}, the result is a classical positive integral closed frieze of width $m$ and period $n=m+3$. The Conway--Coxeter theorem therefore determines a unique triangulation of the labeled $n$-gon. Denote it by $\Delta(\cF)$.

\begin{lemma}[Monomial skeleton]\label[lemma]{lem:skeleton}
Let $\cF$ be a normalized positive Laurent decorated frieze. In the fundamental domain $\mathscr D_n$, 
an interior entry $W_{i,j}$ is a Laurent monomial if and only if the chord $[v_i,v_j]$ is a diagonal of $\Delta(\cF)$. In that case $W_{i,j}=x_k$ for a unique $k$.
\end{lemma}

\begin{proof}
If $W_{i,j}$ is a monomial, normalization gives $W_{i,j}=x_k$, hence $\ev(W_{i,j})=1$. In the classical Conway--Coxeter labeling attached to a triangulation, the interior entries equal to $1$ occur precisely at the triangulation diagonals. Therefore $[v_i,v_j]$ belongs to $\Delta(\cF)$.

Conversely, suppose $[v_i,v_j]$ is a diagonal of $\Delta(\cF)$. Then $\ev(W_{i,j})=1$. Since $W_{i,j}$ is a Laurent polynomial with nonnegative integral coefficients, its value at $(1,\ldots,1)$ is the sum of its coefficients. This sum can be $1$ only when $W_{i,j}$ consists of a single Laurent monomial with coefficient $1$. By normalization it is one of the $x_k$. Since there are exactly $m$ diagonals and all $x_k$ occur, the assignment of the $x_k$ to the diagonals is bijective.
\end{proof}

\begin{corollary}
%[Monomial quiddity entries]
\label[corollary]{cor:monomial}
If $m\geq1$, the quiddity sequence contains at least two entries which are variables $x_k$.
\end{corollary}

\begin{proof}
Every triangulated $n$-gon with $n>3$ has at least two ears \cite[Lemma~1]{BCI}. If an ear has tip $v_{i+1}$, its opposite chord $[v_i,v_{i+2}]$ is a triangulation diagonal. The corresponding quiddity entry is $a_i=W_{i,i+2}$, so \cref{lem:skeleton} gives $a_i=x_k$.
\end{proof}

A weaker form, sufficient for induction, can also be obtained without first identifying the whole skeleton. If all variables are specialized to $1$, the classical quiddity contains an entry $1$; positivity then forces the corresponding decorated quiddity entry to be a monomial, and normalization makes it one of the $x_k$. This is the decorated replacement for the classical fact that a positive closed quiddity contains a $1$.

\section{Weighted triangulations and a geometric quiddity formula}\label{sec:geometry}
Let $\Delta$ be an $(\bx,\by)$-weighted triangulation of the labeled convex $n$-gon $P_n$. Thus the boundary edge $[v_i,v_{i+1}]$ has weight $y_i$, while the $m=n-3$ diagonals of $\Delta$ are labeled bijectively by $x_1,\ldots,x_m$.

For a vertex $v$, let $E_v$ be the set of all triangulation edges incident to $v$, and let $D_v\subset E_v$ be the set of internal diagonals incident to $v$. For a triangle $T$ containing $v$, denote by $E_v(T)$ the two sides of $T$ incident to $v$ and 
by $\operatorname{opp}_v(T)$ the side of $T$ opposite $v$.

\begin{definition}\label{def:Q}
Define
\begin{equation}\label{eq:geometric-quiddity}
 Q_\Delta(v)=
 \frac{
 \displaystyle\sum_{\substack{T\in\Delta,\ v\in T}}
 \wt\bigl(\operatorname{opp}_v(T)\bigr)
 \displaystyle\prod_{e\in E_v\setminus E_v(T)}\wt(e)
 }{
 \displaystyle\prod_{d\in D_v}\wt(d)
 },
\end{equation}
where $\wt(\cdot)$ is the weight on an edge.
The empty product is assumed to be $1$. The \emph{geometric quiddity sequence} is indexed according to the frieze convention by
\begin{equation}\label{eq:Q-index}
 a_i=Q_\Delta(v_{i+1}).
\end{equation}
\end{definition}

If all weights are $1$, each summand in the numerator is $1$ and the denominator is $1$, so $Q_\Delta(v)$ is the number of triangles incident to $v$. Thus \eqref{eq:geometric-quiddity} is a direct weighted extension of the classical quiddity rule.

\begin{figure}[h]
    \centering
    \begin{tikzpicture}[scale=1.2]
        % マーカーのスタイル定義
        \tikzset{
            marker/.style={circle, draw, inner sep=0.5pt, minimum size=10pt, font=\scriptsize\bfseries},
            red_m/.style={marker, draw=red, text=red},
            blue_m/.style={marker, draw=blue, text=blue},
            green_m/.style={marker, draw=green!60!black, text=green!60!black},
            black_m/.style={marker, draw=black, text=black}
        }

        % 七角形の頂点座標 (v1~v7)
        \foreach \i in {1,2,3,4,5,6,7} {
            \coordinate (v\i) at ({90 + 360/7 - (\i-1)*360/7}:2);
        }

        % --- 辺と対角線の描画 ---

        % 外周 (y)
        \draw (v1) -- node[above left] {$y_1$} (v2);
        \draw (v2) -- node[above right] {$y_2$} (v3);
        \draw (v3) -- node[right] {$y_3$} (v4);
        \draw (v4) -- node[below right] {$y_4$} (v5);
        \draw (v5) -- node[below] {$y_5$} (v6);
        \draw (v6) -- node[below left] {$y_6$} (v7);
        \draw (v7) -- node[left] {$y_7$} (v1);
        
        % v5-v7 (x2)
        \draw (v5) -- node[below] {$x_2$} (v7);
        \draw (v5) -- (v6) -- (v7);

        % 対角線 from v2 (x)
        \draw (v2) -- node[pos=0.4, right] {$x_1$} (v7);
        \draw (v2) -- node[pos=0.4, left] {$x_3$} (v5);
        \draw (v2) -- node[pos=0.5, left] {$x_4$} (v4);

        % --- 頂点の描画 ---
        \foreach \i in {1,2,3,4,5,7} {
            \fill (v\i) circle (2pt) node[label={90 + 360/7 - (\i-1)*360/7:$v_{\i}$}] {};
        }
        %\fill[gray] (v6) circle (1.5pt) node[label=below:$v_6$] {};
        \fill (v6) circle (2pt) node[label=below:$v_6$] {};

        % --- マーカーの配置 (calcライブラリ必須) ---

        % Edge y1 (v1-v2): \UTF{2460}, \UTF{2462}, \UTF{2463}%
        \node[red_m] at ($(v1)!0.3!(v2) + (-0.15, 0)$) {1};
        \node[green_m] at ($(v1)!0.5!(v2) + (-0.15, 0)$) {3};
        \node[black_m] at ($(v1)!0.7!(v2) + (-0.15, 0)$) {4};

        % Edge y2 (v2-v3): \UTF{2460}, \UTF{2461}, \UTF{2462}%
        \node[red_m] at ($(v2)!0.3!(v3) + (0.15, -0.1)$) {1};
        \node[blue_m] at ($(v2)!0.5!(v3) + (0.15, -0.1)$) {2};
        \node[green_m] at ($(v2)!0.7!(v3) + (0.15, -0.1)$) {3};

        % Edge y3 (v3-v4): \UTF{2463}%
        \node[black_m] at ($(v3)!0.5!(v4) + (0, 0.4)$) {4};

        % Edge y4 (v4-v5): \UTF{2462}%
        \node[green_m] at ($(v4)!0.5!(v5) + (0.15, 0.25)$) {3};

        % Diagonal x2 (v5-v7): \UTF{2460}%
        \node[red_m] at ($(v5)!0.5!(v7) + (-0.3,0.1)$) {1};

        % Edge y7 (v7-v1): \UTF{2461}%
        \node[blue_m] at ($(v7)!0.5!(v1) + (0.1, 0.2)$) {2};

        % Diagonal x1 (v2-v7): \UTF{2462}, \UTF{2463}%
        \node[green_m] at ($(v2)!0.7!(v7) + (0.1, 0.25)$) {3};
        \node[black_m] at ($(v2)!0.6!(v7) + (0.1, 0.25)$) {4};

        % Diagonal x3 (v2-v5): \UTF{2461}, \UTF{2463}%
        \node[blue_m] at ($(v2)!0.7!(v5) + (-0.1, 0.1)$) {2};
        \node[black_m] at ($(v2)!0.6!(v5) + (-0.1, 0.1)$) {4};

        % Diagonal x4 (v2-v4): \UTF{2460}, \UTF{2461}%
        \node[red_m] at ($(v2)!0.8!(v4) + (-0.2, 0.1)$) {1};
        \node[blue_m] at ($(v2)!0.7!(v4) + (-0.2, 0.1)$) {2};

    \end{tikzpicture}
    \caption{The formula $a_1=Q_\Delta(v_2)$ for vertex $v_2$. For each circled number, take the product of the weights of edges
    where the number is put. Then sum all such products and divide it by $x_1x_3x_4$.}
    \label{fig:numerator_composition}
\end{figure}
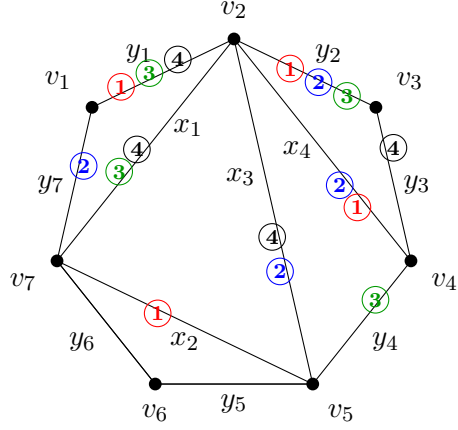

\begin{example}
Let $\Delta$ be the weighted triangulation of the labeled convex heptagon as in Figure \ref{fig:numerator_composition}.
The first geometric quiddity is 
\[
a_1=Q_\Delta(v_2)= \frac{x_2 x_4 y_1 y_2 + x_3 x_4 y_2 y_7 + x_1 y_1 y_2 y_4 + x_1 x_3 y_1 y_3}{x_1 x_3 x_4}.
\]
This is chosen as follows. The denominator is the product
$\prod_{d\in D_{v_2}}\wt(d)$ of weights $x_1,x_3,x_4$ on the internal diagonals incident to $v_2$.
For the numerator, choose four triangles $T$ containing $v_2$, 
and the weights $\wt(\operatorname{opp}_v(T))$ on the side of $T$ opposite $v_2$. For example,
side $\operatorname{opp}_{v_2}(T)=[v_5,v_7]$ of triangle $T=[v_2,v_5,v_7]$ has weight $x_2$ 
marked by $\textcircled{\scriptsize 1}$.
For this fixed $T$, the product $\prod_{e\in E_v\setminus E_v(T)}\wt(e)$ is made of weights $x_4$, $y_1$, and $y_2$
also marked by $\textcircled{\scriptsize 1}$. For the other three triangles $T$, the same construction is
applied.
\end{example}

\begin{lemma}
%[Ear value]
\label[lemma]{lem:ear}
Suppose a triangle $[v_i,v_{i+1},v_{i+2}]$ is an ear with tip $v_{i+1}$ and opposite diagonal $[v_i,v_{i+2}]$ of weight $x$. Then
\[
 a_i=Q_\Delta(v_{i+1})=x.
\]
\end{lemma}

\begin{proof}
Exactly one triangle is incident to $v_{i+1}$, no internal diagonal is incident to the tip, and the side opposite the tip has weight $x$. Formula \eqref{eq:geometric-quiddity} gives the result.
\end{proof}

\begin{proposition}[Ear attachment]\label{prop:ear-attach}
Suppose that in a weighted triangulated $(n-1)$-gon the boundary edge $[v_i,v_{i+2}]$ has weight $x$. Attach the ear $[v_i,v_{i+1},v_{i+2}]$, with new boundary weights $y_i$ and $y_{i+1}$. If $\ba'$ and $\ba$ denote the geometric quiddity sequences 
given by \eqref{eq:Q-index}
before and after the attachment, then
\begin{align}
 a_{i-1}&=\frac{a_{i-1}'y_i+y_{i-1}y_{i+1}}{x},\label{eq:attach-left}\\
 a_i&=x,\label{eq:attach-mid}\\
 a_{i+1}&=\frac{a_i'y_{i+1}+y_i y_{i+2}}{x},\label{eq:attach-right}
\end{align}
and all other entries are unchanged up to the index shift.
\end{proposition}

\begin{figure}[h]
\centering
\begin{tikzpicture}[scale=0.93,every node/.style={font=\small}]

%------------------------------------------------
% Before attachment
%------------------------------------------------
\begin{scope}
\coordinate (L0) at (0.6,0);
\coordinate (A)  at (0.5,1);
\coordinate (B)  at (1.2,2.2);
\coordinate (D)  at (3.5,2.2);
\coordinate (E)  at (4.2,1);
\coordinate (R0) at (4,0);

% Partial boundary of the polygon
\draw (L0)--(A)--(B)--(D)--(E)--(R0);

% Vertex labels: indices increase clockwise
\fill (A) circle (2pt) node[label=below left:$v_{i-1}$] {};
\fill (B) circle (2pt) node[label=above left:$v_{i}$] {};
\fill (D) circle (2pt) node[label=above right:$v_{i+2}$] {};
\fill (E) circle (2pt) node[label=below right:$v_{i+3}$] {};
%\node[below left] at (A) {$v_{i-1}$};
%\node[above left] at (B) {$v_i$};
%\node[above right] at (D) {$v_{i+2}$};
%\node[below right] at (E) {$v_{i+3}$};

% Edge labels
\node[left] at ($(A)!0.5!(B)$) {$y_{i-1}$};
\node[above] at ($(B)!0.5!(D)$) {$x$};
\node[right] at ($(D)!0.5!(E)$) {$y_{i+2}$};

\node[below] at (2.40,-0.55)
{before attachment};
\end{scope}

%------------------------------------------------
% Arrow
%------------------------------------------------
\draw[-{Latex[length=2.5mm]},thick]
      (6.15,0.85)--(7.35,0.85);
\node[above] at (6.75,0.85) {attach an ear};

%------------------------------------------------
% After attachment
%------------------------------------------------
\begin{scope}[xshift=8.6cm]
\coordinate (L1) at (0.6,0);
\coordinate (A1) at (0.5,1);
\coordinate (B1) at (1.2,2.2);
\coordinate (C1) at (2.2,2.8);
\coordinate (D1) at (3.5,2.2);;
\coordinate (E1) at (4.2,1);
\coordinate (R1) at (4,0);

% Partial boundary after the ear is attached
\draw (L1)--(A1)--(B1)--(C1)--(D1)--(E1)--(R1);

% The old boundary edge becomes an internal diagonal
\draw[thick] (B1)--(D1);

% Vertex labels: indices increase clockwise
\fill (A1) circle (2pt) node[label=below left:$v_{i-1}$] {};
\fill (B1) circle (2pt) node[label=above left:$v_{i}$] {};
\fill (C1) circle (2pt) node[label=above:$v_{i+1}$] {};
\fill (D1) circle (2pt) node[label=above right:$v_{i+2}$] {};
\fill (E1) circle (2pt) node[label=below right:$v_{i+3}$] {};
%\node[below left] at (A1) {$v_{i-1}$};
%\node[above left] at (B1) {$v_i$};
%\node[above] at (C1) {$v_{i+1}$};
%\node[above right] at (D1) {$v_{i+2}$};
%\node[below right] at (E1) {$v_{i+3}$};

% Edge labels
\node[left] at ($(A1)!0.5!(B1)$) {$y_{i-1}$};
\node[above left] at ($(B1)!0.5!(C1)$) {$y_i$};
\node[above right] at ($(C1)!0.5!(D1)$) {$y_{i+1}$};
\node[right] at ($(D1)!0.5!(E1)$) {$y_{i+2}$};
\node[below] at ($(B1)!0.5!(D1)$) {$x$};

\node[below] at (2.40,-0.55)
{after attachment};
\end{scope}

\end{tikzpicture}
\caption{Ear attachment.  Before the attachment, the edge
$[v_i,v_{i+2}]$ of weight $x$ is a boundary edge.
Attaching the new vertex $v_{i+1}$ with boundary weights
$y_i$ and $y_{i+1}$ turns this edge into an internal diagonal.}
\label{fig:ear-attachment}
\end{figure}
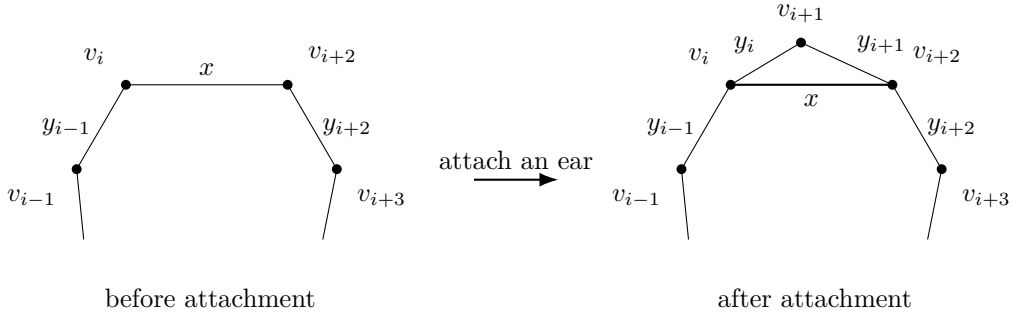

\begin{proof}
Let $\Delta'$ denote the triangulation before the ear is attached and let $\Delta$ denote the triangulation afterwards. Only the stars of the three vertices $v_i,v_{i+1},v_{i+2}$ change. The value at the new ear tip $v_{i+1}$ is immediately
$Q_\Delta(v_{i+1})=x$
by \cref{lem:ear}, which gives \eqref{eq:attach-mid}. We give the bookkeeping at the two endpoints of the new ear in detail.

First consider $v_i$. Write the numerator and denominator of the geometric formula for $\Delta'$ at $v_i$ as
\[
 N_i'
 =\sum_{\substack{T\in\Delta',\,v_i\in T}}
 \wt\bigl(\operatorname{opp}_{v_i}(T)\bigr)
 \prod_{e\in E_{v_i}(\Delta')\setminus E_{v_i}(T)}\wt(e),
 \qquad
 D_i'=\prod_{d\in D_{v_i}(\Delta')}\wt(d).
\]
Thus
\[
 a_{i-1}'=Q_{\Delta'}(v_i)=\frac{N_i'}{D_i'}.
\]
In $\Delta'$, the edge $[v_i,v_{i+2}]$ of weight $x$ is a boundary edge. After the ear is attached it becomes an internal diagonal, while the new boundary edge $[v_i,v_{i+1}]$ of weight $y_i$ is added to the star of $v_i$. Consequently the denominator acquires exactly one new factor:
\[
 D_i=xD_i'.
\]

We next determine the new numerator $N_i$. There are two kinds of contributions. Every triangle of $\Delta'$ incident to $v_i$ remains a triangle of $\Delta$. For each such old triangle, the new edge $[v_i,v_{i+1}]$ is incident to $v_i$ but is not a side of that triangle. Hence the product in its numerator term acquires the extra factor $y_i$. The total contribution of all old triangles is therefore
$y_iN_i'$.
There is in addition one new term, coming from the attached ear
$T_{\rm ear}=[v_i,v_{i+1},v_{i+2}]$.
At the vertex $v_i$, the side opposite $v_i$ is $[v_{i+1},v_{i+2}]$, of weight $y_{i+1}$. The two sides of $T_{\rm ear}$ incident to $v_i$ have weights $y_i$ and $x$. After these two edges are omitted from the incident-edge product in \eqref{eq:geometric-quiddity}, what remains is the old boundary edge $[v_{i-1},v_i]$ of weight $y_{i-1}$ together with all old diagonals incident to $v_i$. Hence the new ear contributes $y_{i+1}y_{i-1}D_i'$.
It follows that
\[
 N_i=y_iN_i'+y_{i+1}y_{i-1}D_i'.
\]

Dividing by $D_i=xD_i'$ gives
\[
 Q_\Delta(v_i)
 =\frac{y_iN_i'+y_{i+1}y_{i-1}D_i'}{xD_i'}
 =\frac{a_{i-1}'y_i+y_{i-1}y_{i+1}}{x},
\]
which is \eqref{eq:attach-left}.

The endpoint $v_{i+2}$ is handled in the same way.
%and we spell it out because the two terms in the update formula have a transparent geometric origin. 
Put
\[
 N_{i+2}'
 =\sum_{\substack{T\in\Delta',\,v_{i+2}\in T}}
 \wt\bigl(\operatorname{opp}_{v_{i+2}}(T)\bigr)
 \prod_{e\in E_{v_{i+2}}(\Delta')\setminus E_{v_{i+2}}(T)}\wt(e),
 \quad
 D_{i+2}'=\prod_{d\in D_{v_{i+2}}(\Delta')}\wt(d),
\]
so that $a_i'=Q_{\Delta'}(v_{i+2})=N_{i+2}'/D_{i+2}'$. Again the former boundary edge of weight $x$ becomes a diagonal, and therefore the new denominator is $xD_{i+2}'$. Every old numerator term acquires the new incident boundary weight $y_{i+1}$, giving $y_{i+1}N_{i+2}'$. For the new ear, the side opposite $v_{i+2}$ is $[v_i,v_{i+1}]$ of weight $y_i$, while after removing the two ear sides incident to $v_{i+2}$ the remaining incident edges consist of the old boundary edge $[v_{i+2},v_{i+3}]$ of weight $y_{i+2}$ and all old diagonals. Thus the new ear contributes $y_i y_{i+2}D_{i+2}'$. We obtain
\[
 Q_\Delta(v_{i+2})
 =\frac{y_{i+1}N_{i+2}'+y_i y_{i+2}D_{i+2}'}{xD_{i+2}'}
 =\frac{a_i'y_{i+1}+y_i y_{i+2}}{x},
\]
which is \eqref{eq:attach-right}.

Finally, the star of every other vertex is unchanged by the ear attachment, so its value in \eqref{eq:geometric-quiddity} is unchanged, apart from the evident relabeling of indices. This completes the proof.
\end{proof}

Solving \eqref{eq:attach-left} and \eqref{eq:attach-right} for the reduced values gives
\begin{equation}\label{eq:geo-reduction}
 a_{i-1}'=\frac{a_{i-1}x-y_{i-1}y_{i+1}}{y_i},
 \qquad
 a_i'=\frac{xa_{i+1}-y_i y_{i+2}}{y_{i+1}}.
\end{equation}
These are regarded as the formulae for ear removal.
These are exactly the decorated replacements for subtracting $1$ 
from the two neighbors of an ear entry in the classical quiddity sequence.

\section{Algebraic reduction of the decorated frieze}\label{sec:reduction}
Let $\cF$ be a normalized positive Laurent decorated frieze of width $m\geq1$. 
By \cref{cor:monomial}, after a cyclic relabeling we may suppose
\[
 a_i=x
\]
for one of the variables $x=x_k$.

Remove $x$ from the list of diagonal variables. Define the new boundary sequence $\by'=(y_1',\ldots,y_{n-1}')$ by
\begin{equation*}\label{eq:y-red}
 y_r'=\begin{cases}
 y_r,&1\leq r\leq i-1,\\
 x,&r=i,\\
 y_{r+1},&i+1\leq r\leq n-1,
 \end{cases}
\end{equation*}
and define the reduced quiddity sequence $\ba'=(a_1',\ldots,a_{n-1}')$ by
\begin{equation}\label{eq:a-red}
 a_r'=\begin{cases}
 a_r,&1\leq r\leq i-2,\\[1mm]
 \displaystyle\frac{a_{i-1}x-y_{i-1}y_{i+1}}{y_i},&r=i-1,\\[3mm]
 \displaystyle\frac{xa_{i+1}-y_i y_{i+2}}{y_{i+1}},&r=i,\\[3mm]
 a_{r+1},&i+1\leq r\leq n-1
 \end{cases}
\end{equation}
making use of \eqref{eq:geo-reduction}.
%This transformation is visualized in Figure \ref{fig:reduced}.

The transfer matrix used in the recurrence is
\[
 M(a;u,v)=
 \begin{pmatrix}
 a/u&-v/u\\
 1&0
 \end{pmatrix}.
\]

\begin{lemma}
%[Transfer-matrix identity]
\label[lemma]{lem:matrix}
With $a_i=x$ and $a_{i-1}',a_i'$ defined by \eqref{eq:a-red},
\[
 M(a_{i+1};y_{i+1},y_{i+2})
 M(x;y_i,y_{i+1})
 M(a_{i-1};y_{i-1},y_i)
 =
 M(a_i';x,y_{i+2})
 M(a_{i-1}';y_{i-1},x).
\]
%\begin{align}
%&M(a_{i+1};y_{i+1},y_{i+2})
% M(x;y_i,y_{i+1})
% M(a_{i-1};y_{i-1},y_i)\notag\\
%&\qquad=
% M(a_i';x,y_{i+2})
% M(a_{i-1}';y_{i-1},x).
%\label{eq:matrix-id}
%\end{align}
\end{lemma}

%\begin{proof}
%Substitute the two formulae in \eqref{eq:a-red} into the right-hand side and multiply the $2\times2$ matrices. The four entries agree with the product on the left. This is the decorated analogue of the elementary three-matrix-to-two-matrix identity used in the classical cutting proof at a quiddity entry $1$.
%\end{proof}

\begin{proof}
We verify the identity by multiplying the matrices explicitly.  The point of
the computation is that the product of the three transfer matrices around the
entry $a_i=x$ is exactly replaced by the product of the two transfer matrices
associated with the reduced data.

For brevity, put
\[
 L=
 M(a_{i+1};y_{i+1},y_{i+2})
 M(x;y_i,y_{i+1})
 M(a_{i-1};y_{i-1},y_i)
\]
and
\[
 R=
 M(a_i';x,y_{i+2})
 M(a_{i-1}';y_{i-1},x).
\]
We first compute $L$.  Multiplying the two matrices on the right gives
\begin{align*}
&M(x;y_i,y_{i+1})
 M(a_{i-1};y_{i-1},y_i)\\
&\qquad=
\begin{pmatrix}
\dfrac{x}{y_i} & -\dfrac{y_{i+1}}{y_i}\\[6pt]
1&0
\end{pmatrix}
\begin{pmatrix}
\dfrac{a_{i-1}}{y_{i-1}} & -\dfrac{y_i}{y_{i-1}}\\[6pt]
1&0
\end{pmatrix}
=
\begin{pmatrix}
\dfrac{x a_{i-1}-y_{i-1}y_{i+1}}
      {y_{i-1}y_i}
&
-\dfrac{x}{y_{i-1}}\\[8pt]
\dfrac{a_{i-1}}{y_{i-1}}
&
-\dfrac{y_i}{y_{i-1}}
\end{pmatrix}.
\end{align*}
Hence
\begin{equation}\label{eq:three-product}
\begin{aligned}
L
&=
\begin{pmatrix}
\dfrac{a_{i+1}}{y_{i+1}}
&
-\dfrac{y_{i+2}}{y_{i+1}}\\[6pt]
1&0
\end{pmatrix}
\begin{pmatrix}
\dfrac{x a_{i-1}-y_{i-1}y_{i+1}}
      {y_{i-1}y_i}
&
-\dfrac{x}{y_{i-1}}\\[8pt]
\dfrac{a_{i-1}}{y_{i-1}}
&
-\dfrac{y_i}{y_{i-1}}
\end{pmatrix}\\
&=
\begin{pmatrix}
\dfrac{
 a_{i+1}(x a_{i-1}-y_{i-1}y_{i+1})
 -y_i y_{i+2}a_{i-1}}
 {y_{i-1}y_i y_{i+1}}
&
\dfrac{-x a_{i+1}+y_i y_{i+2}}
 {y_{i-1}y_{i+1}}
\\[12pt]
\dfrac{x a_{i-1}-y_{i-1}y_{i+1}}
 {y_{i-1}y_i}
&
-\dfrac{x}{y_{i-1}}
\end{pmatrix}.
\end{aligned}
\end{equation}

We next compute the product for the reduced data. By definition,
\[
 a_{i-1}'
 =
 \frac{x a_{i-1}-y_{i-1}y_{i+1}}{y_i},
 \qquad
 a_i'
 =
 \frac{x a_{i+1}-y_i y_{i+2}}{y_{i+1}}.
\]
Therefore
\begin{equation}\label{eq:two-product}
R=
\begin{pmatrix}
\dfrac{a_i'}{x}
&
-\dfrac{y_{i+2}}{x}\\[6pt]
1&0
\end{pmatrix}
\begin{pmatrix}
\dfrac{a_{i-1}'}{y_{i-1}}
&
-\dfrac{x}{y_{i-1}}\\[6pt]
1&0
\end{pmatrix}
=
\begin{pmatrix}
\dfrac{a_i'a_{i-1}'-y_{i-1}y_{i+2}}
      {x y_{i-1}}
&
-\dfrac{a_i'}{y_{i-1}}\\[10pt]
\dfrac{a_{i-1}'}{y_{i-1}}
&
-\dfrac{x}{y_{i-1}}
\end{pmatrix}.
\end{equation}
We now compare the four entries of \eqref{eq:three-product} and \eqref{eq:two-product}.  The $(2,2)$
entries are already identical.  For the $(2,1)$ entry, substitution of
the definition of $a_{i-1}'$ gives
\[
 \frac{a_{i-1}'}{y_{i-1}}
 =
 \frac{x a_{i-1}-y_{i-1}y_{i+1}}
      {y_{i-1}y_i},
\]
which is exactly the $(2,1)$ entry of $L$.  Similarly,
\[
 -\frac{a_i'}{y_{i-1}}
 =
 \frac{-x a_{i+1}+y_i y_{i+2}}
      {y_{i-1}y_{i+1}},
\]
so the $(1,2)$ entries also agree.

It remains only to compare the $(1,1)$ entries.  Substituting the two
reduction formulae into the numerator of the $(1,1)$ entry of $R$, we
obtain
\begin{align*}
a_i'a_{i-1}'-y_{i-1}y_{i+2}
&=
\frac{(x a_{i+1}-y_i y_{i+2})
      (x a_{i-1}-y_{i-1}y_{i+1})}
     {y_i y_{i+1}}
-y_{i-1}y_{i+2}\\
&=
\frac{
 (x a_{i+1}-y_i y_{i+2})
 (x a_{i-1}-y_{i-1}y_{i+1})
 -y_{i-1}y_i y_{i+1}y_{i+2}}
 {y_i y_{i+1}}.
\end{align*}
Expanding the numerator, the last two terms cancel:
\begin{align*}
& (x a_{i+1}-y_i y_{i+2})
  (x a_{i-1}-y_{i-1}y_{i+1})
  -y_{i-1}y_i y_{i+1}y_{i+2}\\
&\qquad
=
x^2a_{i+1}a_{i-1}
-xa_{i+1}y_{i-1}y_{i+1}
-xa_{i-1}y_i y_{i+2}\\
&\qquad
=
x\bigl(
x a_{i+1}a_{i-1}
-a_{i+1}y_{i-1}y_{i+1}
-a_{i-1}y_i y_{i+2}
\bigr).
\end{align*}
Consequently,
\begin{align*}
\frac{a_i'a_{i-1}'-y_{i-1}y_{i+2}}
     {x y_{i-1}}
&=
\frac{
x a_{i+1}a_{i-1}
-a_{i+1}y_{i-1}y_{i+1}
-a_{i-1}y_i y_{i+2}}
{y_{i-1}y_i y_{i+1}}\\
&=
\frac{
a_{i+1}(x a_{i-1}-y_{i-1}y_{i+1})
-y_i y_{i+2}a_{i-1}}
{y_{i-1}y_i y_{i+1}},
\end{align*}
which is precisely the $(1,1)$ entry of $L$.

Thus all four entries of the two matrix products agree, and hence the statement is proved.
%\[
% M(a_{i+1};y_{i+1},y_{i+2})
% M(x;y_i,y_{i+1})
% M(a_{i-1};y_{i-1},y_i)
% =
% M(a_i';x,y_{i+2})
% M(a_{i-1}';y_{i-1},x).
%\]
\end{proof}

This explicit calculation shows that the algebraic reduction at the
monomial quiddity entry $a_i=x$ replaces three consecutive transfer
steps by two without changing their total action.

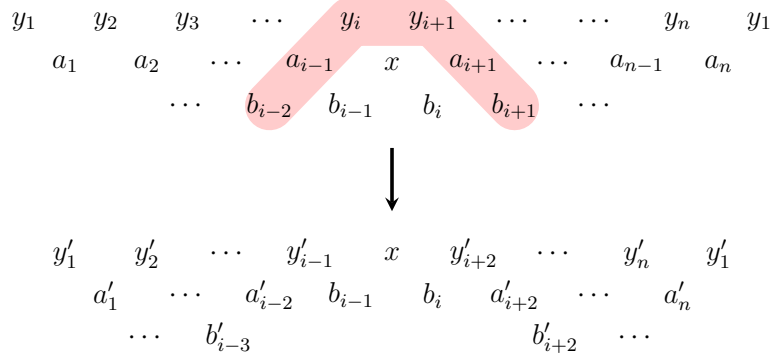
\begin{figure}[h]
\centering
\scalebox{0.90}{%
\begin{tikzpicture}[x=1.2cm,y=0.62cm,font=\large]

  \begin{scope}[on background layer]
    % 1. 左側の斜めライン (b_{i-2} -> a_{i-1} -> y_i)
    % 座標: (2+0.5*2, -2) -> (3+0.5*1, -1) -> (4+0.5*0, 0) つまり (3, -2) -> (3.5, -1) -> (4, 0)
    % 2. 右側の斜めライン (y_{i+1} -> a_{i+1} -> b_{i+1})
    % 座標: (5, 0) -> (5.5, -1) -> (6, -2)
    
    % line width（帯の太さ）を 0.7cm に設定し、角を丸く（line join=round）つなぎます。
    % これにより、文字幅に左右されず、数理的に完全な左右対称の「逆V字の帯」が作れます。
    % 3番目の引数 [red!20] などの部分で色を変更可能です。
    \draw[red!20, line width=0.72cm, line cap=round, line join=round] 
      (3.0, -2.0) -- (4.0, 0.0) -- (5.0, 0.0) -- (6.0, -2.0);
  \end{scope}

  % ==========================================
  % 1つ目の図（上側）
  % ==========================================
  \def\rowsTop{
    {$y_1$, $y_2$, $y_3$, $\cdots$, $y_{i}$, $y_{i+1}$, $\cdots$, $\cdots$, $y_{n}$, $y_1$},
    {$a_1$, $a_2$, $\cdots$, $a_{i-1}$, $x$, $a_{i+1}$, $\cdots$, $a_{n-1}$, $a_{n}$},
    {$\ $, $\cdots$, $b_{i-2}$, $b_{i-1}$, $b_i$, $b_{i+1}$, $\cdots$},
  }
  
  \foreach \row [count=\r from 0] in \rowsTop {
    \foreach \val [count=\c from 0] in \row {
      \node at ({\c+0.5*\r},{-\r}) {\val};
    }
  }

  % ==========================================
  % 中央の下向き矢印
  % ==========================================
  %\node[font=\Large] at (4.5, -3.5) {$\Downarrow$};
  \draw[-stealth, line width=1.5pt] (4.5, -3.0) -- (4.5, -4.5);

  % ==========================================
  % 2つ目の図（下側）
  % ==========================================
  \def\rowsBottom{
    {$\ $},
    {$y'_1$, $y'_2$, $\cdots$, $y'_{i-1}$, $x$, $y'_{i+2}$, $\cdots$, $y'_{n}$, $y'_1$},
    {$a'_1$, $\cdots$, $a'_{i-2}$, $b_{i-1}$, $b_i$, $a'_{i+2}$, $\cdots$, $a'_{n}$},
    {$\cdots$, $b'_{i-3}$, $\ $, $\ $, $\ $, $b'_{i+2}$, $\cdots$},
  }
  \foreach \row [count=\r from 0] in \rowsBottom {
    \foreach \val [count=\c from 0] in \row {
      \node at ({\c+0.5*\r},{-4.5-\r}) {\val};
    }
  }

\end{tikzpicture}%
}
\caption{The reduced decorated frieze after deleting the inverse $V$-strip determined by $a_i=x$. The old diagonal weight $x$ becomes the new boundary decoration $y_i'$. The new quiddity sequence contains $a'_{i-1}=b_{i-1}$ and $a'_i=b_i$. For the other letters, the primes do not change the entries.}
\label{fig:reduced}
\end{figure}

\begin{proposition}[Frieze reduction]\label[proposition]{prop:reduction}
The data $(\by',\ba')$ define a normalized positive Laurent decorated frieze $\cF'$ of width $m-1$. The array $\cF'$ is obtained from $\cF$ by deleting the inverse $V$-shaped strip issued from $a_i=x$ and closing the gap. On the first diagonal,
\begin{equation}\label{eq:shift}
 f_r'=f_r\quad(0\leq r\leq i-2),
 \qquad
 f_r'=f_{r+1}\quad(i-1\leq r\leq m),
\end{equation}
and analogous shifted identities hold on every diagonal.
\end{proposition}

\begin{proof}
Before the junction the boundary and quiddity data are unchanged, so the recurrence gives $f_r'=f_r$. At the junction, the old recurrence passes through the three transfer matrices on the left side of the equation in \cref{lem:matrix}, whereas the reduced recurrence passes through the two matrices on the right. Hence
\[
 \binom{f_i'}{f_{i-1}'}=\binom{f_{i+1}}{f_i}.
\]
Iteration proves the second part of \eqref{eq:shift}. The same argument, with indices shifted, applies to 
the other diagonals and therefore to the entire array.

Thus every reduced entry is a surviving entry of $\cF$, so it remains a Laurent polynomial with nonnegative integral coefficients. The decorated diamond relations away from the junction are inherited from $\cF$, while the relations crossing the junction are encoded by \cref{lem:matrix}. The upper and lower boundary rows close after one fewer interior row, so $\cF'$ is a closed positive Laurent decorated frieze of width $m-1$. Its $(n-1)$-periodicity is then automatic from \cref{thm:periodicity}; it need not be checked separately.

After specializing to $1$, the inverse $V$-deletion is the ordinary removal of the chosen ear from $\Delta(\cF)$. By \cref{lem:skeleton}, the remaining monomial positions are exactly the remaining $m-1$ diagonals, labeled by the variables $\{x_1,\ldots,x_m\}\setminus\{x\}$. Hence $\cF'$ is normalized.
\end{proof}

\begin{remark}\label{rem:CHJ-reduction}
Reduction formulae for quiddity cycles in the general coefficient theory, including a noncommutative version, are known; see \cite[Proposition~6.12]{CHJ2024}. The role of \cref{prop:reduction} here is more specific. It identifies the reduced Laurent-positive array with the inverse $V$-deletion, preserves our normalization, and shows that the algebraic operation is exactly the weighted ear removal needed for the converse classification.
\end{remark}

\section{Proof of the decorated Conway--Coxeter theorem}\label{sec:main}
We first recall the forward construction. Given a weighted triangulation, prescribe the weights $y_i$ on boundary edges and $x_k$ on triangulation diagonals. Ptolemy relations determine the values of all remaining chords. Equivalently, the type $A$ perfect-matching formula gives these values as Laurent polynomials with nonnegative integral coefficients \cite{Propp2020,Nishiyama2022}.

\begin{proposition}\label{prop:forward}
Every $(\bx,\by)$-weighted triangulation $\Delta$ determines a unique normalized positive Laurent decorated frieze $\cF_\Delta$ of width $m=n-3$. Its quiddity sequence is
\[
 a_i=Q_\Delta(v_{i+1}),
\]
where $Q_\Delta$ is given by \eqref{eq:geometric-quiddity}.
\end{proposition}

\begin{proof}
The standard Ptolemy/perfect-matching construction gives a closed positive Laurent decorated frieze with the prescribed boundary and diagonal weights. Its $n$-periodicity follows either geometrically from the cyclic labeling or algebraically from \cref{thm:periodicity}. Uniqueness follows from \cref{cor:unique}.

For normalization, specialize all weights to $1$. The result is the ordinary Conway--Coxeter frieze of the underlying triangulation. An interior entry is $1$ exactly at a triangulation diagonal. Hence a Laurent monomial interior entry must be one of the prescribed diagonal weights $x_k$, while each $x_k$ occurs on its own diagonal.

It remains to prove the quiddity formula. We use induction on $n$. The triangle is immediate. For $n>3$, remove an ear. By induction, the reduced quiddity is given by \eqref{eq:geometric-quiddity}. Reattaching the ear changes the three affected values by \cref{prop:ear-attach}. These are exactly the inverse relations to \eqref{eq:a-red}; hence the quiddity of $\cF_\Delta$ is the geometric sequence in \eqref{eq:Q-index}.
\end{proof}

\begin{theorem}[Decorated Conway--Coxeter classification]\label{thm:main}
The map
\[
 \Delta\longmapsto\cF_\Delta
\]
is a bijection from $(\bx,\by)$-weighted triangulations of the labeled convex $n$-gon onto normalized positive Laurent decorated friezes of width $m=n-3$.
\end{theorem}

\begin{proof}
We prove surjectivity by induction on $m$. The case $m=0$ is a triangle. Let $m\geq1$ and let $\cF$ be normalized. By \cref{thm:periodicity}, $n=m+3$ is a period of $\cF$, so the specialization $\ev(\cF)$ is a classical positive integral closed frieze of width $m$ with an $n$-term quiddity. The classical Conway--Coxeter theorem gives a triangulation $\Delta(\cF)$ of the labeled $n$-gon.

Choose an ear of this triangulation. By \cref{cor:monomial}, after cyclic relabeling its quiddity entry is $a_i=x$ for a distinguished variable $x$. By \cref{prop:reduction}, deleting the corresponding inverse $V$-strip gives a normalized positive Laurent decorated frieze $\cF'$ of width $m-1$. Its period $n-1=m+2$ is supplied automatically by (DCC1). By induction, $\cF'$ comes from a unique weighted triangulation $\Delta'$ of an $(n-1)$-gon.

In the reduced boundary sequence, the new edge has weight $x$. Attach to it an ear with boundary weights $y_i$ and $y_{i+1}$. By \cref{prop:ear-attach}, the quiddity changes by the inverse of \eqref{eq:a-red}. The resulting weighted triangulation therefore has the same boundary and quiddity data as $\cF$. By \cref{cor:unique}, it produces $\cF$.

For injectivity, suppose $\cF_\Delta=\cF_{\widetilde\Delta}$. Specializing to $1$, the classical Conway--Coxeter theorem gives the same underlying labeled triangulation. By \cref{lem:skeleton}, the decorated frieze records the label $x_k$ on every triangulation diagonal. Hence the diagonal labeling also agrees, and $\Delta=\widetilde\Delta$.
\end{proof}

\begin{corollary}[Classical Conway--Coxeter theorem]\label{cor:classical}
Under the specialization
\[
 x_1=\cdots=x_m=y_1=\cdots=y_n=1,
\]
\cref{thm:periodicity,thm:main} reduce respectively to the classical $n$-periodicity property (CC1) and the Conway--Coxeter triangulation correspondence (CC2).
\end{corollary}

\begin{remark}
%[Comparison of the two inductions]
\label{rem:classical-proof}
The proof of \cref{thm:main} is deliberately arranged to mirror the classical Conway--Coxeter induction. In the classical case one finds a quiddity entry $1$, deletes it, decreases its two neighbors by $1$, and removes the corresponding ear of the triangulated polygon. In the decorated case, specialization to $1$ finds the classical ear, Laurent positivity lifts its quiddity entry from $1$ to a monomial $x_k$, and the neighboring entries are modified by \eqref{eq:geo-reduction}. The matrix identity in \cref{lem:matrix} is exactly the decorated replacement for the elementary matrix identity in the classical cutting proof.
\end{remark}

\begin{remark}\label{rem:conceptual}
Once \cref{lem:skeleton} is known, there is also a short conceptual reconstruction: specialization to $1$ recovers the triangulation, normalization recovers the $x_k$-label on each diagonal, and Ptolemy propagation then recovers the entire decorated frieze. We retain the inductive proof because it exhibits the precise algebraic/geometric reduction mechanism and proves the explicit geometric quiddity formula at the same time.
\end{remark}


\begin{thebibliography}{99}

\bibitem{BCI}
D.~Broline, D.~W.~Crowe and I.~M.~Isaacs,
\emph{The geometry of frieze patterns},
Geom. Dedicata \textbf{3} (1974), 171--176.

\bibitem{CH2021}
M.~Cuntz and T.~Holm,
\emph{Subpolygons in Conway--Coxeter frieze patterns},
Algebraic Combinatorics \textbf{4} (2021), no.~4, 741--755.

\bibitem{CHJ2020}
M.~Cuntz, T.~Holm and P.~J\o rgensen,
\emph{Frieze patterns with coefficients},
Forum Math. Sigma \textbf{8} (2020), Paper No.~e17.

\bibitem{CHJ2024}
M.~Cuntz, T.~Holm and P.~J\o rgensen,
\emph{Noncommutative frieze patterns with coefficients},
arXiv:2403.09156.

\bibitem{CC1}
J.~H.~Conway and H.~S.~M.~Coxeter,
\emph{Triangulated polygons and frieze patterns},
Math. Gaz. \textbf{57} (1973), 87--94.

\bibitem{CC2}
J.~H.~Conway and H.~S.~M.~Coxeter,
\emph{Triangulated polygons and frieze patterns (continued)},
Math. Gaz. \textbf{57} (1973), 175--183.

\bibitem{Coxeter1971}
H.~S.~M.~Coxeter,
\emph{Frieze patterns},
Acta Arith. \textbf{18} (1971), 297--310.

\bibitem{Henry2013}
C.-S.~Henry,
\emph{Coxeter friezes and triangulations of polygons},
Amer. Math. Monthly \textbf{120} (2013), no.~6, 553--558.

\bibitem{Morier2Friezes}
S.~Morier-Genoud,
\emph{Arithmetics of $2$-friezes},
J. Algebr. Comb. \textbf{36} (2012), 515--539.

\bibitem{MorierSurvey}
S.~Morier-Genoud,
\emph{Coxeter's frieze patterns at the crossroads of algebra, geometry and combinatorics},
Bull. Lond. Math. Soc. \textbf{47} (2015), 895--938.

\bibitem{MOT2015}
S.~Morier-Genoud, V.~Ovsienko and S.~Tabachnikov,
\emph{$SL_2(\mathbb Z)$-tilings of the torus, Coxeter--Conway friezes and Farey triangulations},
Enseign. Math. \textbf{61} (2015), 71--92.

\bibitem{Nishiyama2022}
K.~Nishiyama,
%\emph{Fr\=izu no s\=ugaku: Sukecchich\=o -- kazu to kika no kirameki}
\emph{A Sketchbook of the Mathematics of Friezes: the Brilliance of Numbers and Geometry},
Kyoritsu Shuppan, Tokyo, 2022 (in Japanese).

\bibitem{Propp2020}
J.~Propp,
\emph{The combinatorics of frieze patterns and Markoff numbers},
Integers \textbf{20} (2020), Paper No.~A12.

\bibitem{Short2025}
I.~Short, M.~van Son and A.~Zabolotskii,
\emph{Frieze patterns and Farey complexes},
Adv. Math. \textbf{472} (2025), Paper No.~110269.

\end{thebibliography}
\end{document}